\documentclass[12pt]{amsart}
\usepackage{geometry}                
\usepackage{graphicx}
\usepackage{amssymb}
\usepackage{epstopdf}
\usepackage{amsmath}
\usepackage{color}
\usepackage{hyperref}
\usepackage{pdfsync}
 \usepackage{tikz}
 \usepackage{tikz-cd}

\usepackage[all]{xy}
\xyoption{matrix}
\xyoption{arrow}

\usetikzlibrary{arrows,decorations.pathmorphing,backgrounds,positioning,fit,petri}

 \tikzset{help lines/.style={step=#1cm,very thin, color=gray},
help lines/.default=.5} 
\tikzset{thick grid/.style={step=#1cm,thick, color=gray},
thick grid/.default=1} 

\newtheorem{thm}{Theorem}[section]
\newtheorem{lem}[thm]{Lemma}
\newtheorem{cor}[thm]{Corollary}
\newtheorem{prop}[thm]{Proposition}

 \newtheorem{innercustomthm}{Theorem}
\newenvironment{customthm}[1]
  {\renewcommand\theinnercustomthm{#1}\innercustomthm}{\endinnercustomthm}

\theoremstyle{definition}
\newtheorem{defn}[thm]{Definition}

\theoremstyle{remark}

\numberwithin{equation}{section}

\newcommand{\into}{\hookrightarrow}
 \newcommand{\onto}{\twoheadrightarrow}

 \DeclareMathOperator{\Diff}{Diff}

\newcommand{\field}[1]{\mathbb{#1}}
\newcommand{\ZZ}{\ensuremath{{\field{Z}}}}

\newcommand{\RR}{\ensuremath{{\field{R}}}}

\newcommand{\commentout}[1]{}

\newcommand{\cC}{\ensuremath{{\mathcal{C}}}}
\newcommand{\cD}{\ensuremath{{\mathcal{D}}}}

\newcommand{\cP}{\ensuremath{{\mathcal{P}}}}

\title{Second obstruction to pseudoisotopy in dimension 3}
\author{Kiyoshi Igusa}
\address{Department of Mathematics, Brandeis University, Waltham, MA 02454}\email{igusa@brandeis.edu}
\thanks{Supported by the Simons Foundation}

\subjclass[2020]{
19J10: 58K60}

\begin{document}

\begin{abstract}
We use lens-shaped models and the second obstruction to pseudoisotopy to construct a nontrivial diffeomorphism of $M\times I$ where $M$ is the connected sum of $S^1\times S^2$ with another nonsimply connected 3-manifold $M'$. Then we take two copies of this diffeomorphism and paste together their tops and bottoms to obtain a diffeomorphism of $M\times S^1$. Properties of the second obstruction and the first Postnikov invariant imply that this diffeomorphism of the closed 4-manifold $M\times S^1$ is not isotopic to the identity. Similar results were obtain by Singh \cite{S}.
\end{abstract}

\maketitle



\section*{Introduction}

There has been a resurgence of interest in diffeomorphisms of 4-manifolds, for example the work of Watanabe \cite{W}, \cite{W1} and Budney-Gabai \cite{BG}. Also, Singh \cite{S} has constructed diffeomorphisms of 4-manifolds using the methods of Hatcher and Wagoner \cite{HW}, \cite{H2}.

In this paper we also use Hatcher-Wagoner pseudoisotopy theory to construct pseudoisotopies of a family of 3-manifolds $M$ and use these to construct nontrivial diffeomorphisms of the closed $4$-manifolds $M\times S^1$. Recall that a pseudoisotopy of $M$ is a diffeomorphism of $M\times I$ which is the identity on $M\times 0$. The space of pseudoisotopies of $M$ is denoted $\cC(M)$. Our first theorem is:

\begin{customthm}{A}[Theorem \ref{theorem A}]\label{thm A}
Let $M'$ be any nonsimply connected $3$-manifold and let $M=(S^1\times S^2)\# M'$. Then there is a pseudoisotopy of $M$ every power of which is nontrivial.
\end{customthm}

To prove this we use ``lens-shaped models'' to construct diffeomorphisms of $M\times I$ which are the identity on the bottom ($M\times 0$) and show using our formulas from \cite{What happens} that these diffeomorphisms are nontrivial. More precisely, we compute the ``second obstruction'' which is an element of the group $Wh_1^+(\pi_1M;\ZZ_2\oplus \pi_2M)$. In the first example of this (when $M'=S^1\times S^2$) the first Postnikov invariant of $M$ is trivial (since $\pi_1M$ is a free group which has no cohomology in degrees $\ge2$). So this preliminary calculation is enough to prove Theorem \ref{thm A} in this case.

In the general case, when $M'$ is an arbitrary nontrivial 3-manifold, we need a calculation of the mapping
\[
	\chi: K_3(\ZZ[\pi_1M])\to Wh_1^+(\pi_1M;\ZZ_2\oplus\pi_2M)
\]
induced by the first Postnikov invariant $k_1M$ (ignoring the $\ZZ_2$ component $Wh_1^+(\pi_1M;\ZZ_2)$.

We use the fact that the sphere $S^2$ in $S^1\times S^2$ is a retract of $M$ and we use the retraction $r:M\to S^2$ to get a surjective mapping
\[
	r_\ast: Wh_1^+(\pi_1M;\pi_2M)\to Wh_1^+(\pi_1M;\ZZ).
\]
We show that the composition $r_\ast\circ \chi=0$. Therefore, when $r_\ast$ is applied to the second obstruction of our example, the result is nonzero. So, our pseudoisotopy is stably nontrivial.

At this point we need to discuss the stablization process. We need a definition of the second obstruction invariant which commutes with stabilization:
\[
	\cC(M)\to \cP(M)=colim\, \cC(M\times I^n)
\]
given by iterating the suspension map $\sigma_\pm:\cC(M)\to \cC(M\times I)$ (Figures \ref{Fig: positive suspension}, \ref{Fig: negative suspension}). However, positive suspension and negative suspension are negative to each other. So, we need our second obstruction to commute with $\sigma_+$ and anti-commute with $\sigma_-$. To this end we need to introduce a sign, namely $(-1)^k$ where $k$ is the lower index of the two indices used in the lens-shaped model. Then we have a well-defined stable invariant given as follows.

First, we need the definition of the stable first invariant. This is given by stabilizing, then taking the first obstruction:
\[
	\pi_0\cC(M)\to \pi_0\cP(M)\to Wh_2(\pi_1M).
\]
Then we define the stable second invariant on the kernel of the stable first invariant.

The reason that we need to make this fuss is because of the involution on lens-shaped models. Since $3$ is an odd number, the space of lens-shaped models cannot be chosen to be invariant under the involution. In fact the involution sends $\cD^-(M)$ to $\cD^+(M)$ where $\cD^-(M)$ is the space of lens-shaped models in the indices $1,2$ and $\cD^+(M)$ is the same in indices $2,3$. 

We have trouble computing the second obstruction on the sum $g\cup \varepsilon g$ since the two summands are elements of different groups $\pi_0\cD_0^-(M)$ and $\pi_0\cD_0^+(M)$. (We use marked lens-shaped models so that they form a group.) However, the stable second obstruction is additive (it makes sense to add their values on the two pieces even though they lie in different groups). We show that this is well defined and give a ``stable retraction invariant'' in $Wh_1^+(\pi_1M;\ZZ)$. We compute this on $g\cup \varepsilon g$ and show that this ($g$ union the upside-down version $\varepsilon g$) is stably nontrivial. This gives a nontrivial diffeomorphism of $M\times I$ fixing the boundary. 

By a general fact which should be well-known (but I don't know where to find it) any such diffeomorphism gives a nontrivial diffeomorphism of $M\times S^1$.

Finally we note that the diffeomorphism $g\cup \varepsilon g$ of $M\times I$ and the resulting diffeomorphism of $M\times S^1$ are both pseudoisotopic to the identity. The pseudoisotopy is given by the positive suspension of $g$. Since the obstruction group $Wh_1^+(\pi_1M;\ZZ)$ is a free abelian group we obtain the following.

\begin{customthm}{B}\label{thm B}
For $M=(S^1\times S^2)\#M'$ as above, there is a diffeomorphism of $M\times S^1$ every power of which is nontrivial. Furthermore this diffeomorphism is pseudoisotopic to the identity.
\end{customthm}

The paper is organized as follows. In Section 1 we review the definition of $Wh_1^+(G;A)$ for any $G$-module $A$. In Section 2 we construct the specific lens-shaped model for $M=(S^1\times S^2)\#M'$ using any nontrivial element of $\pi_1M'$. (Figure \ref{Fig: snake figure to construct ft}). Section 3 discusses the Postnikov invariant $k_1M$ and its affect on the second obstruction. We construct the ``retraction invariant'' and show it is well-defined. This proves Theorem \ref{thm A}. Section 4 deals with the involution $\varepsilon$ which turns a pseudoisotopy upside-down. We also show how to stabilize the second obstruction so we can add the invariants for $g$ and $\varepsilon g$. Section 5 goes over the ``closing the clam'' construction (Figure \ref{Fig: closing the clam}) to complete the proof of Theorem \ref{thm B}.

The author would like to thank Danny Ruberman for asking him if the 4-manifold pseudoisotopy constructed in \cite{Wh1a} gives a nontrivial diffeomorphism of the top $M\times 1$. This paper is an attempt to answer that question. The author also acknowledges support from the Simons Foundation.

\section{The second obstruction group}

We review the algebra of the second obstruction group $Wh_1^+(\pi_1M;\ZZ_2\oplus \pi_2M)$.

\begin{defn}
For any group $G$ and left $G$-module $A$ let $A[G]$ be the $G$-module $A\otimes \ZZ[G]$ with diagonal action action of $G$ where the action of $G$ on $\ZZ[G]$ is by conjugation:
\[
	g(a\otimes h)=ga\otimes ghg^{-1}.
\]
The group of coinvariants of this action is $A[G]_G=H_0(G;A[G])$. This is also isomorphic to $\ZZ[G]\otimes_GA$ if $G$ acts on $\ZZ[G]$ on the right by conjugation: $(\sum n_ig_i)\cdot h=\sum n_i\, h^{-1}gh$. $A[1]$ is a submodule of $A[G]$ and its coinvariants give 
\[
	Wh_1^+(G;A):=(A[G]/A[1])_G=H_0(G;A[G])/H_0(G;A).
\]
\end{defn}

For example, when the action of $G$ on $A$ is trivial, $Wh_1^+(G;A)$ is the direct sum of copies of $A$, one for every conjugacy class of nonidentity elements of $G$. 

In particular, one can show that the element $\alpha[\sigma]\in A[G]$ gives a nontrivial element of $Wh_1^+(G;A)$ if $\sigma$ is not the identity in $G$ and if there is a homomorphism $\varphi:A\to B$ where $B$ has trivial $G$ action so that $\varphi(\alpha)\neq0$. In that case 
\[
\varphi_\ast:Wh_1^+(G;A)\to Wh_1^+(G;B)
\]
sends $\alpha[\sigma]$ to $\varphi(\alpha)[\sigma]$ which is nontrivial in $Wh_1^+(G;B)$.

%
%

\section{Constructing the lens-shaped model} 

Let $M$ be the connected sum of $S^1\times S^2$ with another nontrivial $3$-manifold $M'$. Let $\alpha\in \pi_2M$ be given by the 2-sphere $S^2\subset S^1\times S^2$ and let $\sigma$ be any nontrivial element of $\pi_1M'\subset \pi_1M$. We will construct a lens-shaped model for $M$ with second obstruction $\alpha[\sigma]$.

Recall that a \emph{lens-shaped model} for $M$ is a 1-parameter family of functions $f_t:M\times I$ whose graphic is a 1-lens. (See Figure \ref{Fig: 1-lens}.) We refer to \cite{Wh1a} for definitions. We note however that there is one important difference between dimensions 3 and 4. In dimension 3 there are two choices for the middle two indices. We take our lens-shaped model to be a family of functions $f_t$ with critical points in indices $1,2$. Let $\cD^-(M)$ denote the space of all such lens-shaped models. Let $\cD_0^-(M)$ be the space of ``marked'' lens-shaped models. (See \cite{Wh1a}.) The purpose of the marking is to make $\pi_0\cD_0^-(M)$ into a group. To show that multiplication is well-defined we need only to observe that $O(4)/O(2)$ is simply connected. The marking has no affect on the second obstruction if we ignore the framing invariant.

\begin{figure}[htbp]
\begin{center}
\begin{tikzpicture}[scale=1.2]
\draw (-3.7,-1.5) rectangle (3.7,1.5);
\draw[thick] (-2.8,0).. controls(-1.8,0) and (-1.4,-1)..(0,-1)..controls (1.4,-1) and (1.8,0)..(2.8,0);
\draw[thick] (-2.8,0).. controls(-1.8,0) and (-1.4,1)..(0,1)..controls (1.4,1) and (1.8,0)..(2.8,0);
\begin{scope}
\draw[thick,dashed,gray] (-2.8,0).. controls(-1.8,0) and (-.7,-.4)..(0,-.4)..controls (.7,-.4) and (1.8,0)..(2.8,0);
\draw[very thick,dashed,red] (0,.0) ellipse [x radius=4mm,y radius=2mm];
\draw[thick,blue,dotted] (0,-1)--(-.4,0)--(0,1)--(.4,0)--(0,-1); 
\draw[blue] (0,.5) node{$\alpha$};
\end{scope}
\draw[thick] (-2.8,-1.4)--(-2.8,-1.6) node[below] {$t_0$};
\draw[thick] (0,-1.4)--(0,-1.6) node[below] {$t_1$};
\draw[thick] (2.8,-1.4)--(2.8,-1.6) node[below] {$t_2$};
\draw (-1.7,.8) node{$q_t$};
\draw (-1.7,-.8) node{$p_t$};
\end{tikzpicture}
\caption{Depicted is the graphic of a ``1-lens'', a 1-parameter family of functions $f_t$ representing an pseudoisotopy of a 3-manifold $M$. The points $p_t,q_t$ are critical of $f_t$ of indices 1,2. The dashed red circle and dotted blue 2-sphere are explained in the text.}
\label{Fig: 1-lens}
\end{center}
\end{figure}
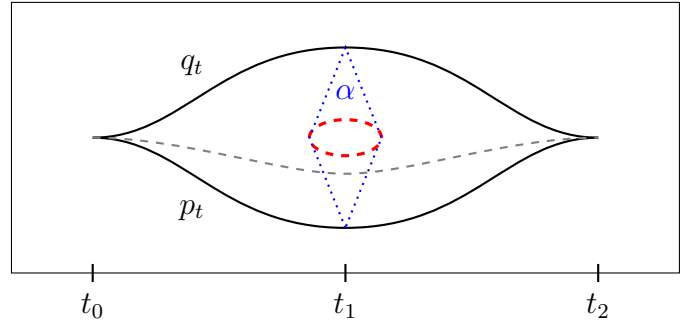

Figure \ref{Fig: 1-lens} gives the graphic of the lens-shaped model we will construct. Recall that this is the set of ordered pairs $(t,s)$ where $s$ is a critical value of $f_t$ for $t\in I$. The function $f_0:M\times I\to I$ is the projection map. At $t=t_0$ the function has a birth-death point. For $t$ slightly more than $t_0$, we have a Morse function $f_t:M\times I\to I$ with two canceling critical points $p_t,q_t$ of indices $1,2$ and there is a single trajectory of the gradient of $f_t$ going from $p_t$ to $q_t$. Along the deformation, two additional trajectories are created then canceled and these trajectories form a circle of trajectories spanning a $2$-sphere as indicated in Figure \ref{Fig: 1-lens}. For $t_0<t<t_2$, the intermediate level surface $f_t^{-1}(\frac12)$ is $V^3$ which is the connected sum of $M$ with another $S^1\times S^2$ as shown in Figure \ref{Fig: snake figure to construct ft}.

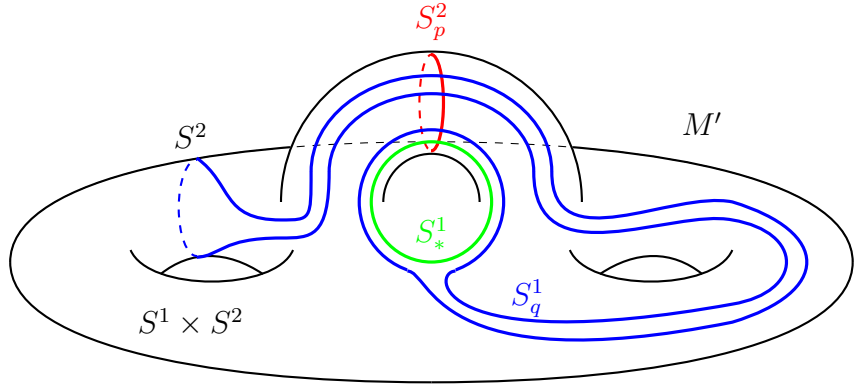
\begin{figure}[htbp]
\begin{center}
\begin{tikzpicture}[scale=.8]
\begin{scope} 
\draw[ thick] (0,2)..controls (4,2) and (7,1.3)..(7,0); 
\draw[ thick] (2.8,-.2)..controls (3.4,.2) and (4,.2)..(4.5,-.2);
\draw[ thick] (2.3,.2)..controls (2.6,-.5) and (4.7,-.5)..(5,.2);
\draw[ thick] (0,-2)..controls (4,-2) and (7,-1.3)..(7,0);
\end{scope}
\begin{scope}[xscale=-1] 
\draw[ thick] (0,2)..controls (4,2) and (7,1.3)..(7,0); 
\draw[ thick] (2.8,-.2)..controls (3.4,.2) and (4,.2)..(4.5,-.2);
\draw[ thick] (2.3,.2)..controls (2.6,-.5) and (4.7,-.5)..(5,.2);
\draw[ thick] (0,-2)..controls (4,-2) and (7,-1.3)..(7,0);
\end{scope}
\begin{scope} 
\clip (-3,1) rectangle (3,4);
\draw[ thick,fill,white] (0,1) circle[radius=2.5cm];
\draw[ thick] (0,1) circle[radius=8mm];
\draw[ thick] (0,1) circle[radius=2.5cm];
\end{scope}
\draw[dashed] (0,2)..controls (4,2) and (7,1.3)..(7,0); 
\begin{scope}[xscale=-1] 
\draw[dashed] (0,2)..controls (4,2) and (7,1.3)..(7,0); 
\end{scope}
\begin{scope}[yshift=-1mm] 
\draw[dashed,thick,red] (0,2.75) ellipse[x radius=2mm, y radius=8mm];
\clip (0,1) rectangle (.5,3.6);
\draw[very thick,red] (0,2.75) ellipse[x radius=2mm, y radius=8mm];
\end{scope}
\draw[very thick,blue] (0,1) circle[radius=1.2];
\draw[fill,white] (0,-.2) circle[radius=.4];
\draw[very thick,blue] (-.4,-.15)..controls (0,-.15) and (0,-2)..(5,-1)..controls (6.65,-.5) and (6.65,.5)..(5,1)..controls (4,1.2) and (2,0)..(2,1.5);
\draw[very thick,blue] (.4,-.15)..controls (0,-.15) and (0,-1.6)..(5,-.7)
..controls (6.2,-.5) and (6.2,.5)..(5,.7)..controls (4,.9) and (1.7,-.3)..(1.7,1.5);
\begin{scope}
\clip (-2.5,1.5) rectangle (2.5,3.2);
\draw[very thick, blue] (0,1.5) ellipse[x radius=1.7,y radius=1.3];
\draw[very thick, blue] (0,1.5) ellipse[x radius=2,y radius=1.6];
\end{scope}
\draw[very thick,blue] (-1.7,1.5) ..controls (-1.7,.4) and (-2,.4)..(-2.5,.4)..controls (-3.5,.4) and (-3.5,0.1)..(-3.9,0.08);
\draw[very thick,blue] (-2,1.5) ..controls (-2,.7) and (-2,.7)..(-2.5,.7)..controls (-3.5,.7) and (-3.3,1.3)..(-3.9,1.73);
\begin{scope}
\clip (-3.9,2)rectangle(-4.7,0);
\draw[ thick,blue,dashed] (-3.9,.9)ellipse[x radius=3mm,y radius=8mm];
\end{scope}
\draw (-4,2.1) node{$S^2$};
\draw (-4,-1) node{$S^1\times S^2$};
\draw (4.5,2.3) node{$M'$};
\draw[red] (0,4) node{$S^2_p$};
\draw[blue] (1.6,-.6) node{$S^1_q$};
\draw[very thick,green] (0,1) circle[radius=1] (0,0) node[above]{$S^1_\ast$};
\end{tikzpicture}
\caption{The level surface of the Morse function $f_t$ for $t=t_1$ (in the middle of the 1-lens in Figure \ref{Fig: 1-lens}) is shown with the stable 1-sphere $S^1_q$ of the index 2 critical point $q_t$ in blue (deformed from the standard circle $S_\ast^1$ in green) and the unstable 2-sphere $S^2_p$ of the index 1 critical point $p_t$ in red. These cross in three points. }
\label{Fig: snake figure to construct ft}
\end{center}
\end{figure}

The 1-parameter family of functions $f_t$ on $M\times I$ is constructed as follows. The function has one critical point $p_t$ of index 1. This attaches a 1-handle to $M\times [0,1/4]$ to produce the level surface $M\# (S^1\times S_p^2)$ with the meridian 2-sphere $S^2_p$ being the unstable sphere of $p_t$. This 2-sphere is shown in red in Figure \ref{Fig: snake figure to construct ft}. The longitudinal 1-sphere $S^1_\ast$ meets $S^2_p$ transversely in one point. This shown in green in Figure \ref{Fig: snake figure to construct ft}. 

For $t=t_0+\epsilon$ (a little bit past the birth point of $f_t$ at $t=t_0$), the stable sphere of $q_t$ will be the green circle $S^1_\ast$ in Figure \ref{Fig: snake figure to construct ft}. This is in ``cancelling position'' with $S^2_p$. As $t$ goes from $t_0$ to $t_2$, we deform this 2-sphere in the level surface as follows. We push one part of it through an embedded loop representing any nontrivial $\sigma\in \pi_1M'$ on the right side of Figure \ref{Fig: snake figure to construct ft}. Then we slide this ``finger'' over the 1-handle created by the critical point $p_t$ increasing to 3 the number of transverse intersection points of $S^2_p$ with $S^1_q$.

Then we wrap this ``finger'' around the 2-sphere in $S^1\times S^2$ as shown in the figure. Then we pull it back. When we pull back the blue circle, the extra two intersections with $S^2_p$ will be eliminated and the blue circle will go back to its original position shown by the green circle $S^1_\ast$ in Figure \ref{Fig: snake figure to construct ft}. Then, the critical points $p_t,q_t$ can cancel at $t=t_2$ and the ``lens'' will be complete. 

\begin{thm}[Theorem \ref{thm A}]\label{theorem A}
The pseudoisotopy of $M=(S^1\times S^2)\# M'$ given by the 1-parameter family of functions $f_t$ constructed above realizes the second obstruction 
\[
\lambda(f_t)=\alpha[\sigma]\in Wh_1^+(\pi_1M;\pi_2M)
\]
where $\alpha\in \pi_2M$ is given by the 2-sphere $S^2$ in $S^1\times S^2$ and $\sigma\in \pi_1M$ comes from a nontrivial element of $\pi_1M'$. This gives a nontrivial element of $\pi_0\cC(M)$.
\end{thm}

\begin{proof}
This follows from the construction. The two additional intersection points of $S_p^2$ with $S_q^1$ give $\pm\sigma$ in the incidence matrix of the Morse complex. The blue circle $S_q^1$ represents the base of a cone of trajectories up to $q_t$. When that base goes through the $2$-sphere in $S^1\times S^2$, that $2$-cycle is pushed into the top of the cone and the circle of trajectories from $p_t$ to $q_t$ will represent that class $\alpha\in \pi_2M$. So, the second obstruction for our lens-shaped model is $\alpha[\sigma]$. When the first Postnikov invariant $k_1M$ is trivial, as in the case $M'=S^1\times S^2$, there is nothing more to do and Theorem \ref{thm A} holds in that case. But, in general, $k_1M\neq0$. Theorem \ref{thm: retraction invariant} will complete the proof.
\end{proof}

\section{Postnikov invariant}

The \emph{first Postnikov invariant} of $M$ is the cohomology class 
\[
	k_1M\in H^3(\pi_1M;\pi_2M)
\]
which is the first obstruction to the existence of a homotopy section of the natural map $M\to B\pi_1M$.
The original formula of Hatcher and Wagoner assumed this invariant to be zero. We recall the formula for what happens when $k_1M$ is nonzero.

\begin{thm}\cite{What happens}
In the stable range ($\dim M\ge 6$) there is an exact sequence
\begin{equation}\label{eq: what happens}
	Wh_3(\pi_1M)\xrightarrow\chi Wh_1^+(\pi_1M;\ZZ_2\oplus \pi_2M)\to \pi_0\cC(M)\to Wh_2(\pi_1M)\to 0
\end{equation}
where $Wh_3(\pi_1M)$ is some quotient of $K_3(\ZZ[\pi_1M])$ and the second component of the map $\chi$, composed with this quotient map is given by
\[
	K_3(\ZZ[\pi_1M])\to H_3(GL_\infty(\ZZ[\pi_1M])\xrightarrow{\chi_{k_1}}H_0(\pi_1M;\pi_2M[\pi_1M])\to Wh_1^+(\pi_1M;\pi_2M)
\]
where ${\chi_{k_1}}$ is given at the chain level by
\[
	\chi_{k_1}(A,B,C)=\sum_{i,j,k,\ell} f(a_{ij}\otimes b_{jk}\otimes c_{k\ell})[d_{\ell i}]\in \pi_2M[\pi_1M]
\]
for all $A,B,C\in GL_n(\ZZ[\pi_1M])$ where $a_{ij}, b_{jk}, c_{k\ell},d_{\ell i}\in\ZZ[\pi_1M]$ are the entries of $A,B,C$, $D=(ABC)^{-1}$ and $f:\ZZ[\pi_1M]\otimes \ZZ[\pi_1M]\otimes \ZZ[\pi_1M]\to \pi_2M$ is the linearization of the 3-cocycle representing $k_1M$.
\end{thm}

The important aspect of the formula is its naturality:

\begin{cor}
Let $\varphi:\pi_2M\to A$ be a homomorphism of $\pi_1M$ modules and let $k_1A$ be the image of $k_1M$ under the induced map $H^3(\pi_1M;\pi_2M)\to H^3(\pi_1M;A)$. Then the following diagram commutes
\[
\xymatrixrowsep{10pt}\xymatrixcolsep{60pt}
\xymatrix{
& H_0(\pi_1M;\pi_2M[\pi_1M])\ar[dd]^{\varphi_\ast}\\
H_3(GL_\infty(\ZZ[\pi_1M])\ar[ru]^{\chi_{k_1M}}\ar[rd]^{\chi_{k_1A}} \\
& H_0(\pi_1M;A[\pi_1M])
	}
\]
\end{cor}

We use the naturality of the Postnikov invariant to show that the map $\chi$ does not hit our two second obstruction elements $\alpha[\sigma]$ and $\alpha[\sigma]+\alpha[\sigma^{-1}]$ (which will occur later).

The connected sum $M=S^1\times S^2\#M'$ has separating $2$-sphere which cuts $M$ into two $3$-manifolds with boundary $S^2$, call them $W,W'$. Thus $W$ is $S^1\times S^2$ minus a $3$-ball and $W'$ is $M'$ minus a $3$-ball and $M/W'=S^1\times S^2$. Let $X=W'\vee S^1$. This embeds in $M$ with the same fundamental group $\pi_1X\cong \pi_1M=\pi$. By naturality of the Postnikov invariant, the map in cohomology
\[
	j_\ast:H^3(\pi_1X;\pi_2X)\to H^3(\pi_1M;\pi_2M)
\]
induced by the inclusion map $j:X\into M$ sends $k_1X$ to $k_1M$.

The inclusion map $S^2\to M$ has a retraction $r:M\to S^2$ given by first pinching $W'\subset M$ to a point, then projecting to the factor $S^2$:
\[
	M=(S^1\times S^2)\# M'\to S^1\times S^2\to S^2.
\]
The induced map on $\pi_2$, $\pi_2M\to \pi_2S^2=\ZZ$ is $\pi_1M$ equivariant since it factors through the Hurewicz map $\pi_2M\to H_2(M)$. Since this retraction sends $X$ to one point, the induced map in cohomology 
\[
r_\ast: H^3(\pi_1M;\pi_2M)\to H^3(\pi_1M;\ZZ)
\]
given by the coefficient map $\pi_2M\to \pi_2S^2=\ZZ$ induced by $r$ sends $k_1M$ to $0$.

The naturality argument is the following commuting diagram.
\[
\xymatrix{
K_3(\ZZ[\pi_1M])\otimes H^3(\pi_1M;\pi_2M)\ar[d]^{id\otimes r_\ast}\ar[r] & H_0(\pi_1M;\pi_2M[\pi_1M])\ar[d]\ar[r]& Wh_1^+(\pi_1M;\pi_2M)\ar[d]\\
K_3(\ZZ[\pi_1M])\otimes H^3(\pi_1M;\ZZ)\ar[r] & H_0(\pi_1M;\ZZ[\pi_1M])\ar[r]& 
Wh_1^+(\pi_1M;\ZZ)
	}
\]
Since $k_1M$ goes to $0$ in $H^3(\pi_1M;\ZZ)$, the image of $K_3(\ZZ[\pi_1M])$ in $Wh_1^+(\pi_1M;\pi_2M)$ goes to zero in $Wh_1^+(\pi_1M;\ZZ)$. Since the second obstruction element $\alpha[\sigma]\in Wh_1^+(\pi_1M;\pi_2M)$ goes to $1[\sigma]\in Wh_1^+(\pi_1M;\ZZ)$, it survives to $\pi_0\cC(M)$ showing that the corresponding pseudoisotopy is nontrivial and, furthermore, gives a nontrivial element of $\pi_0\cP(M)$.

The group $Wh_1^+(\pi_1M;\ZZ)$ is the free abelian group generated by the set of conjugacy classes of nontrivial elements of $\pi_1M$. Thus, $\alpha[\sigma]$ and $\alpha[\sigma]+\alpha[\sigma^{-1}]$ map to $[\sigma]$ and $[\sigma]+[\sigma^{-1}]$ which are both nonzero in $Wh_1^+(\pi_1M;\ZZ)$ even in the case when $\sigma$ is conjugate to its inverse in which case we would get $[\sigma]+[\sigma^{-1}]=2[\sigma]$.

\begin{thm}\label{thm: retraction invariant}
The second obstruction elements $\alpha[\sigma]$ and $\alpha[\sigma]+\alpha[\sigma^{-1}]$ are not in the image of $\chi:Wh_3(\pi_1M)\to Wh_1^+(\pi_1M;\ZZ_2\oplus \pi_2M)$ and therefore survive to nontrivial elements of $\pi_0\cP(M)$. In particular, our construction gives two nontrivial elements of $\pi_0\cC(M)$.
\end{thm}

\section{Involution and suspension}

For a pseudoisotopy $g\in \cC(M)$ we have an involution $\varepsilon$ which acts by
\[
	\varepsilon(g)=(r(g)\times id_I)^{-1}\circ\tau\circ g\circ\tau
\] 
where $\tau$ is the automorphism of $M\times I$ given by $\tau(x,t)=(x,1-t)$ and $r(g)$ is the restriction of $g$ to $M\times 1$. On the corresponding family of functions $f_t:M\times I\to I$ (with $f_0$ the identity and $f_1=p_2\circ g$, $p_2$ being projection to $I$), the involution $\varepsilon$ acts by 
\[
\varepsilon(f_t)(x,s)=1-f_t(x,1-s).
\]
\begin{lem}
The pseudoisotopy $g\circ \varepsilon(g)$ is the identity on both top and bottom of $M\times I$ and, after expanding $I$ to $[0,2]$, $g\circ \varepsilon(g)$ is isotopic to $g\cup \tau\circ g\circ \tau$.
\end{lem}

\begin{proof}
When expanding $I$ to $[0,2]$, we need to extend $g$ to $M\times [1,2]$ by $r(g)\times id_{[1,2]}$. We shift $\varepsilon(g)$ to $M\times [1,2]$ and extend to $M\times I$ by the identity. Then $r(g)\times id$ cancels $(r(g)\times id)^{-1}$ and the composition becomes $g$ on $M\times I$ and $\tau\circ g\circ \tau$ on $M\times [1,2]$.
\end{proof}

We denote $g\cup \tau\circ g\circ \tau$ by $g\cup \varepsilon(g)$ even though it is not quite correct. The corresponding family of functions is $f_t\cup \varepsilon(f_t)$ which is a union of two lenses, one in indices $1,2$ and the other in indices $2,3$. These lie in two different groups: $\pi_0\cD_0^-(M)$ and $\pi_0\cD_0^+(M)$ where $\cD_0^+(M)$ is the space of marked lens-shaped models in indices $2,3$. To make the homomorphism
\[
 \varepsilon:\pi_0\cD_0^-(M)\to\pi_0\cD_0^+(M)
\]
well-defined on the markings, we take the orientation at the birth point of a marked lens-shaped model in $\cD_0^-(M)$, pull it back to the base point of $M\times I^2$, take the complementary orientation using a fixed orientation of the tangent space at the base point, then push it to the base point of the upside-down lens-shaped model in $\cD_0^+(M)$. This uses the fact that both $O(4)/O(2)$ and $O(4)/SO(2)\times O(1)$ are simply connected.

We need to find compatible definitions of the second invariant on $\cD_0^+(M)$ and $\cD_0^-(M)$ in order to show that $g\cup\varepsilon(g)$ is nontrivial.

\subsection{Suspension}

We recall the (positive) suspension operator
\[
	\sigma_+:\cC(M)\to \cC(M\times [-1,1])
\]
which is given in modified polar coordinates $[r,\theta]\in [-1,1]\times [0,\pi]$ (explained below) by
\[
	\sigma_+(g)(x,r,\theta)=(g(x,r),\theta)
\]
where the relation to standard coordinates in $I\times [-1,1]$ is
\[
	(r,\theta)\leftrightarrow (0,1)-((1-r)\cos\theta,(1-r)\sin\theta).
\]
We assume that $g$ is a diffeomorphism of $M\times I$ which is the identity near the bottom and sides and equal to $r(g)\times id_I$ near the top. We extend $g$ to $M\times [-1,1]$ by taking $g$ to be the identity on $M\times [-1,0]$. Figure \ref{Fig: positive suspension} is the standard visualization of the suspension. The following follows directly from this description.

\begin{prop}
The top of the concordance $\sigma_+(g)$ is the diffeomorphism of $M\times [-1,1]$ given by $g$ (shifted down by 1) on $M\times [-1,0]$ and $\varepsilon g$ on $M\times [0,1]$. In particular, $g\cup  \varepsilon g$ is pseudoisotopic to the identity on $M\times [-1,1]$. See Figure \ref{Fig: positive suspension}.
\end{prop}

\begin{figure}[htbp]
\begin{center}
\begin{tikzpicture}
\begin{scope}
\clip (-2,-2)rectangle(2,0);
\draw[fill,gray!15!white] (-2,-2)rectangle(2,0);
\draw[fill,white] (0,0)circle[radius=2cm];
\end{scope}
\draw (-2,-2)rectangle(2,0);
\begin{scope}
\begin{scope}[rotate=180]
\draw[thick] (0,0)--(1,0); 
\draw[thick,<-] (1,0)--(2,0);
\end{scope}
\begin{scope}[rotate=0]
\draw[thick] (0,0)--(1,0); 
\draw[thick,<-] (1,0)--(2,0);
\end{scope}
\begin{scope}[rotate=-15]
\draw[thick] (0,0)--(1,0); 
\draw[thick,<-] (1,0)--(2,0);
\end{scope}
\begin{scope}[rotate=-30]
\draw[thick] (0,0)--(1,0); 
\draw[thick,<-] (1,0)--(2,0);
\end{scope}
\begin{scope}[rotate=-45]
\draw[thick] (0,0)--(1,0); 
\draw[thick,<-] (1,0)--(2,0);
\end{scope}
\begin{scope}[rotate=-60]
\draw[thick] (0,0)--(1,0); 
\draw[thick,<-] (1,0)--(2,0);
\end{scope}
\begin{scope}[rotate=-75]
\draw[thick] (0,0)--(1,0); 
\draw[thick,<-] (1,0)--(2,0);
\end{scope}
\end{scope}
\begin{scope}[rotate=-90]
\begin{scope}[rotate=0]
\draw[thick] (0,0)--(1,0); 
\draw[thick,<-] (1,0)--(2,0);
\end{scope}
\begin{scope}[rotate=-15]
\draw[thick] (0,0)--(1,0); 
\draw[thick,<-] (1,0)--(2,0);
\end{scope}
\begin{scope}[rotate=-30]
\draw[thick] (0,0)--(1,0); 
\draw[thick,<-] (1,0)--(2,0);
\end{scope}
\begin{scope}[rotate=-45]
\draw[thick] (0,0)--(1,0); 
\draw[thick,<-] (1,0)--(2,0);
\end{scope}
\begin{scope}[rotate=-60]
\draw[thick] (0,0)--(1,0); 
\draw[thick,<-] (1,0)--(2,0);
\end{scope}
\begin{scope}[rotate=-75]
\draw[thick] (0,0)--(1,0); 
\draw[thick,<-] (1,0)--(2,0);
\end{scope}
\end{scope}
\end{tikzpicture}
\caption{The positive suspension $\sigma_+(g)$ is given by taking each arrow, considered as a copy of $M\times I$, applying the function $g$, then putting it back. It is the identity on the shaded region. The diffeomorphism at the top is $g$ and the upside-down version of $g$, i.e., $g\cup\varepsilon(g)$.}
\label{Fig: positive suspension}
\end{center}
\end{figure}
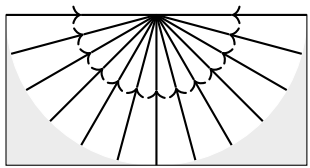

\subsection{The stable retraction invariant}  In order to define the second obstruction on $g\cup \varepsilon(g)$ we need to define it in a compatible way on $\cD^-(M)$ and $\cD^+(M)$. This is done by multiplying by a sign. We also need to compose with the retraction map $Wh_1^+(\pi_1M;\ZZ_2\oplus \pi_2M)\to Wh_1^+(\pi_1M;\ZZ)$ to make sure it is well defined. The result will be the ``{stable retraction invariant}''.

Let $\cD_0^k(M)$ denote the space of marked lens-shaped models for an $n$-manifold $M$ in indices $k,k+1$ where $n\ge 3$ and $k,n-k\ge 1$ so that $O(n+1)/SO(k)\times O(n-k)$ is simply connected making $\pi_0\cD_0^k(M)$ into a group. We define the \emph{stable second obstruction} to be homomorphism $\lambda_0^s=(-1)^k\lambda_0$ where
\[
	\lambda_0:\pi_0\cD_0^k(M)\to Wh_1^+(\pi_1M;\ZZ_2\oplus \pi_2M)
\]
is the standard second obstruction homomorphism. This invariant has the property of being compatible with positive suspension $\sigma_+:\cD_0^k(M)\to \cD_0^{k}(M\times I)$ since $\lambda_0$ commutes with $\sigma_+$, but $\lambda_0^s$ is anti-compatible with $\sigma_-:\cD_0^k(M)\to \cD_0^{k+1}(M\times I)$ since $\sigma_-$ changes the parity of $k$: $\lambda_0^s\circ \sigma_-=-\lambda_0^s$. This sign convention makes $\lambda_0^s$ compatible with stabilization:
\[
	\cP(M)=colim\, \cC(M\times I^m)
\]
since stabilization is taken by iterating the positive suspension $\sigma_+$. It is well known that $\sigma_-$ is equal to $-\sigma_+$ on the level of homotopy groups $\sigma_-=-\sigma_+:\pi_0\cC(M)\to \pi_0\cC(M\times I)$. 

For the case at hand, we have the following diagram.
\[
\xymatrix{
\pi_0\cD_0^+(M)=\pi_0\cD_0^2(M)\ar[r]^(.6){\sigma_+}\ar@/_1pc/[rr]_{\lambda_0^s} & \pi_0\cD_0^2(M\times I)\ar[r]^(.4){\lambda_0^s}& Wh_1^+(\pi_1M;\ZZ_2\oplus \pi_2M)\\
\pi_0\cD_0^-(M)=\pi_0\cD_0^1(M)\ar[r]^(.6){-\sigma_-}\ar[u]^{\varepsilon} \ar@/_1pc/[rr]_{\lambda_0^s} & \pi_0\cD_0^2(M\times I)\ar[r]^(.4){\lambda_0^s}& Wh_1^+(\pi_1M;\ZZ_2\oplus \pi_2M)
	}
\]

\begin{lem} Given that $\lambda_0(g)=\alpha[\sigma]$ for $g\in \cD_0^-(M)$, we obtain: $\lambda_0^s(g)=-\alpha[\sigma]$ and $\lambda^s_0(\varepsilon(g))=-\alpha[\sigma^{-1}]$.
\end{lem}

\begin{proof}
As in the case of 4-manifolds \cite{Wh1a}, the involution changes the second obstruction of $g$ to $\lambda_0(\varepsilon(g)=-\alpha[\sigma^{-1}]$. However, the involution also changes the index of the critical points as indicated in the diagram above. So, $\varepsilon(g)\in \cD_0^+(M)$. The stable invariant $\lambda_0^s$ keeps the same sign on $\pi_0\cD_0^+(M)$ givinb $\lambda^s_0(\varepsilon(g)=-\alpha[\sigma^{-1}]$. However, $\lambda_0^s$ changes the sign of $\lambda_0(g)$. So, $\lambda_0^s(g)=-\alpha[\sigma]$.
\end{proof}

When the first Postnikov invariant $k_1M=0$ we can use this lemma to compute the second obstruction for $g\cup\varepsilon(g)$:
\[
	\lambda_0^s(g\cup \varepsilon(g))=-\lambda_0(\sigma_-g)+\lambda_0\varepsilon(\sigma_+g)=-\alpha[\sigma]-\alpha[\sigma^{-1}]\neq0.
\]
In the general case we need to stabilize and use a well-defined ``stable retraction invariant'' defined as follows.

For $M=S^1\times S^2\#M'$ as above, let
\[
	\rho: \pi_0\cD_0^k(M\times I^m)\to Wh_1^+(\pi_1M;\ZZ)
\]
be the homomorphism given by $\rho(g)=r_\ast(\lambda_0^s(g))$ where $r_\ast:Wh_1^+(\pi_1M;\ZZ_2\oplus \pi_2M)\to Wh_1^+(\pi_1M;\ZZ)$ is the map induced by the retraction $r:M\to S^2$. By the exact sequence
\[
	\pi_0\cD_0^k(M\times I^m)\to \pi_0\cP(M)\to Wh_2(\pi_1M)\to 0
\]
and the fact that $r_\ast$ is zero on the image of $\chi:Wh_3(\pi)\to Wh_1^+(\pi_1M;\ZZ_2\oplus \pi_2M)$, $\rho$ gives a well defined homomorphism from the kernel of the first invariant $\pi_0\cP(M)\to Wh_2(\pi_1M)$ to $Wh_1^+(\pi_1M;\ZZ)$. We call this the \emph{stable retraction invariant}.

\begin{thm}
The stable retraction invariant is realized on the $3$-manifold $M=S^1\times S^2\#M'$ and, on our examples $g$ and $g\cup \varepsilon(g)$ takes the nonzero values
\[
	\rho(g)=-[\sigma],\quad \rho(g\cup \varepsilon(g))=-[\sigma]-[\sigma^{-1}].
\]
\end{thm}

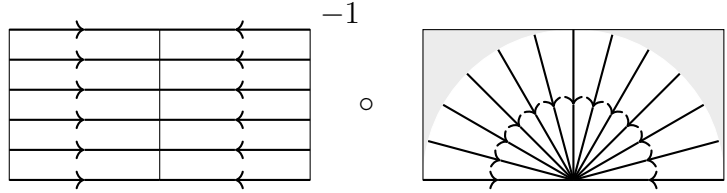
\begin{figure}[htbp]
\begin{center}
\begin{tikzpicture}
\begin{scope}[xshift= -55mm] 
\draw (-2,0)rectangle(2,2);
\draw (0,0)--(0,2);
\draw (2.75,1) node{$\circ$};
\draw (2.4,2.2)node{$-1$};
\begin{scope}[rotate=0]
\draw[thick] (0,0)--(1,0); 
\draw[thick,<-] (1,0)--(2,0);
\draw[thick] (0,0)--(-1,0); 
\draw[thick,<-] (-1,0)--(-2,0);
\end{scope}
\begin{scope}[yshift=4mm]
\draw[thick] (0,0)--(1,0); 
\draw[thick,<-] (1,0)--(2,0);
\draw[thick] (0,0)--(-1,0); 
\draw[thick,<-] (-1,0)--(-2,0);
\end{scope}
\begin{scope}[yshift=8mm]
\draw[thick] (0,0)--(1,0); 
\draw[thick,<-] (1,0)--(2,0);
\draw[thick] (0,0)--(-1,0); 
\draw[thick,<-] (-1,0)--(-2,0);
\end{scope}
\begin{scope}[yshift=12mm]
\draw[thick] (0,0)--(1,0); 
\draw[thick,<-] (1,0)--(2,0);
\draw[thick] (0,0)--(-1,0); 
\draw[thick,<-] (-1,0)--(-2,0);
\end{scope}
\begin{scope}[yshift=16mm]
\draw[thick] (0,0)--(1,0); 
\draw[thick,<-] (1,0)--(2,0);
\draw[thick] (0,0)--(-1,0); 
\draw[thick,<-] (-1,0)--(-2,0);
\end{scope}
\begin{scope}[yshift=2cm]
\draw[thick] (0,0)--(1,0); 
\draw[thick,<-] (1,0)--(2,0);
\draw[thick] (0,0)--(-1,0); 
\draw[thick,<-] (-1,0)--(-2,0);
\end{scope}
\end{scope} 
\begin{scope}[yscale=-1] 
\begin{scope} 
\clip (-2,-2)rectangle(2,0);
\draw[fill,gray!15!white] (-2,-2)rectangle(2,0);
\draw[fill,white] (0,0)circle[radius=2cm];
\end{scope}
\draw (-2,-2)rectangle(2,0);
\begin{scope} 
\begin{scope}[rotate=180]
\draw[thick] (0,0)--(1,0); 
\draw[thick,<-] (1,0)--(2,0);
\end{scope}
\begin{scope}[rotate=0]
\draw[thick] (0,0)--(1,0); 
\draw[thick,<-] (1,0)--(2,0);
\end{scope}
\begin{scope}[rotate=-15]
\draw[thick] (0,0)--(1,0); 
\draw[thick,<-] (1,0)--(2,0);
\end{scope}
\begin{scope}[rotate=-30]
\draw[thick] (0,0)--(1,0); 
\draw[thick,<-] (1,0)--(2,0);
\end{scope}
\begin{scope}[rotate=-45]
\draw[thick] (0,0)--(1,0); 
\draw[thick,<-] (1,0)--(2,0);
\end{scope}
\begin{scope}[rotate=-60]
\draw[thick] (0,0)--(1,0); 
\draw[thick,<-] (1,0)--(2,0);
\end{scope}
\begin{scope}[rotate=-75]
\draw[thick] (0,0)--(1,0); 
\draw[thick,<-] (1,0)--(2,0);
\end{scope}
\end{scope} 
\begin{scope}[rotate=-90] 
\begin{scope}[rotate=0]
\draw[thick] (0,0)--(1,0); 
\draw[thick,<-] (1,0)--(2,0);
\end{scope}
\begin{scope}[rotate=-15]
\draw[thick] (0,0)--(1,0); 
\draw[thick,<-] (1,0)--(2,0);
\end{scope}
\begin{scope}[rotate=-30]
\draw[thick] (0,0)--(1,0); 
\draw[thick,<-] (1,0)--(2,0);
\end{scope}
\begin{scope}[rotate=-45]
\draw[thick] (0,0)--(1,0); 
\draw[thick,<-] (1,0)--(2,0);
\end{scope}
\begin{scope}[rotate=-60]
\draw[thick] (0,0)--(1,0); 
\draw[thick,<-] (1,0)--(2,0);
\end{scope}
\begin{scope}[rotate=-75]
\draw[thick] (0,0)--(1,0); 
\draw[thick,<-] (1,0)--(2,0);
\end{scope}
\end{scope} 
\end{scope} 
\end{tikzpicture}
\caption{The well-known formula $\varepsilon(\sigma_+g)=\sigma_-\varepsilon(g)$ follows from the fact that the outcome of both operations give the same result, illustrated by the figure above. Composition with the inverse of the bottom diffeomorphism appears in the the formula for both $\varepsilon$ and $\sigma_-$.
}
\label{Fig: negative suspension}
\end{center}
\end{figure}
%

%
%

\section{Product with a circle}

Suppose that $g\in \cC(M)$ is the identity on $M\times 1$. Then we can identify top and bottom to obtain a diffeomorphism $\overline g$ of $M\times S^1$. We believe the following is a known argument. We call it ``closing the clam.''

\begin{lem}\label{lemma E} For $M$ a closed manifold
the mapping $\Diff(M\times I, rel\, \partial)\to \Diff(M\times S^1)$ given by identifying top of bottom of $M\times I$ is a monomorphism on components.
\end{lem}

\begin{proof}

Since the map on $\pi_0$ is a homomorphism of groups, it suffices to show that the kernel is zero. Let $g$ be a diffeomorphism of $M\times I$ which is the identity on the boundary and let $\overline g$ be the corresponding diffeomorphism of $M\times S^1$ and suppose $\overline g$ is isotopic to the identity on $M\times S^1$. Then we will construct an isotopy of $g$ to the identity on $M\times I$.

We are given an isotopy of $\overline g$ to the identity on $M\times S^1$. We compose with the map $M\times I\onto M\times S^1$, then lift the composite map to $M\times \RR$. The result is an embedding $\widetilde g:M\times I\to M\times \RR$ which is isotopic to the inclusion map. Moreover, throughout the isotopy $\widetilde g_t$ has the property that, for all $t\in I$, 
\begin{equation}\label{eq: gt is a clam}
	\widetilde g_t(x,1)=\tau \widetilde g_t(x,0)
\end{equation}
where $\tau$ is the deck transformation of $M\times \RR$ over $M\times S^1$ given by $\tau(x,t)=(x,t+1)$.

Since $M\times I$ is compact, we may assume that the image of the isotopy $\widetilde g_t$ stays inside $M\times [-N,N]$ for some large $N$. By the ambient isotopy theorem, this family of embedding $\widetilde g_t$ extends to a family of diffeomorphisms $h_t$ of $M\times [-N,N]$ which is the identity on its boundary for all $t\in I$. By construction, $h_0$ is the identity on $M\times [-N,0]$ and on $M\times[1,N]$. At $t=1$, $h_1$ will be the identity on $M\times I$.

Let $f_t:M\times [-N,0]\to M\times [-N,N]$ and $f_t':M\times [1,N]\to M\times [-N,N]$ be the restrictions of $h_t$ to these two subsets of $M\times[-N,N]$ with images $A_t=im\,f_t$, $B_t=im\,f_t'$. By \eqref{eq: gt is a clam}, the bottom of $B_t$ is the translation $\tau$ of the top of $A_t$. This resembles a ``clam''. (See Figure \ref{Fig: closing the clam}.) When the clam closes, the top and bottom fit neatly together to form 
\[
	A_t\cup \tau^{-1}B_t=M\times [-N,N-1].
\]
The interiors of $A_t$ and $\tau^{-1}B_t$ do not meet since $\widetilde g_t(M\times 0)$ is a separating surface.

An isotopy of $h_0$ to the identity on $M\times [-N,N]$ can now be given by:
\[
	id\cup \widetilde g\cup id \simeq f_t\cup \widetilde g_t\cup f_t'\simeq f_1\cup id \cup f_1'\simeq f_1\cup \tau^{-1} f_1'\tau\cup id\simeq f_t\cup \tau^{-1} f_t'\tau\cup id\simeq id_{M\times[-N,N]}
\]
as shown in Figure \ref{Fig: closing the clam}.
\end{proof}

%
\begin{figure}[htbp]
\begin{center}
\begin{tikzpicture}
\begin{scope}[scale=1.2] 
\draw[ thick] (0,-.1) rectangle (2,3.1);
\draw[very thick,blue] (0,1)--(2,1);
\draw[very thick,blue] (0,2)--(2,2);
\draw[blue] (1,1.5) node{$\widetilde g$};
\draw (1,.5) node{$id$};
\draw (1,2.5) node{$id$};
\end{scope}
\draw (3.2,1.7) node{$\simeq$};
\draw (7.2,1.7) node{$\simeq$};
\draw (11.2,1.7) node{$\simeq$};
\begin{scope}[xshift=4cm,scale=1.2] 
\draw (1,-.5) node{the ``clam''};
\draw[ thick] (0,-.1) rectangle (2,3.1);
\begin{scope}[yshift=-2mm] 
\draw[very thick,blue] (0,1)..controls (2,.7) and (-1,1.4)..(1,1.4)..controls (2,1.4) and (0,1)..(2,1);
\end{scope}
\begin{scope}[yshift=8mm] 
\draw[very thick,blue] (0,1)..controls (2,.7) and (-1,1.4)..(1,1.4)..controls (2,1.4) and (0,1)..(2,1);
\end{scope}
\draw (1,.3) node{$A_t=im\, f_t$};
\draw (1,2.6) node{$B_t=im\, f'_t$};
\draw[blue] (1,1.5) node{$im\,\widetilde g_t$};
\end{scope}
\begin{scope}[xshift=8cm,scale=1.2] 
\draw[ thick] (0,-.1) rectangle (2,3.1);
\draw[very thick,blue] (0,1)--(2,1);
\draw[very thick,blue] (0,2)--(2,2);
\draw (1,.3) node{$ f_1$};
\draw (1,2.6) node{$f'_1$};
\draw[blue] (1,1.5) node{$id_{M\times I}$};
\end{scope}
\begin{scope}[xshift=12cm,scale=1.2] 
\draw (1,-.5) node{closed clam};

\draw[ thick] (0,-.1) rectangle (2,3.1);
\begin{scope}[yshift=-2mm] 
\draw[very thick,blue] (0,1)..controls (2,.7) and (-1,1.4)..(1,1.4)..controls (2,1.4) and (0,1)..(2,1);
\end{scope}
\draw[very thick,blue] (0,2.1)--(2,2.1);
\draw (1,.3) node{$f_t$};
\draw (1,1.6) node{$\tau^{-1}\circ f'_t\circ\tau$};
\draw (1,2.5) node{$id_{M\times [N-1,N]}$};
\end{scope}
\end{tikzpicture}
\caption{The top and bottom portions ($A_t$ and $B_t$) of the diffeomorphism $h_t=f_t\cup \widetilde g_t\cup f_t'$ fit together to give a diffeomorphism $f_t\cup \tau^{-1}f_t'\tau$ of $M\times[-N,N-1]$. The closed clam is isotopic to the identity on $M\times[-N,N]$ since the middle portion $\widetilde g_t$ is missing. Not shown is the isotopy 
\[
f_1\cup id_{M\times I}\cup f_1'\simeq f_1\cup \tau^{-1}\circ f'_1\circ\tau\cup id_{M\times[N-1,N]}
\]
given by sliding $f_1'$ down to $\tau^{-1}f'_1\tau$.}
\label{Fig: closing the clam}
\end{center}
\end{figure}
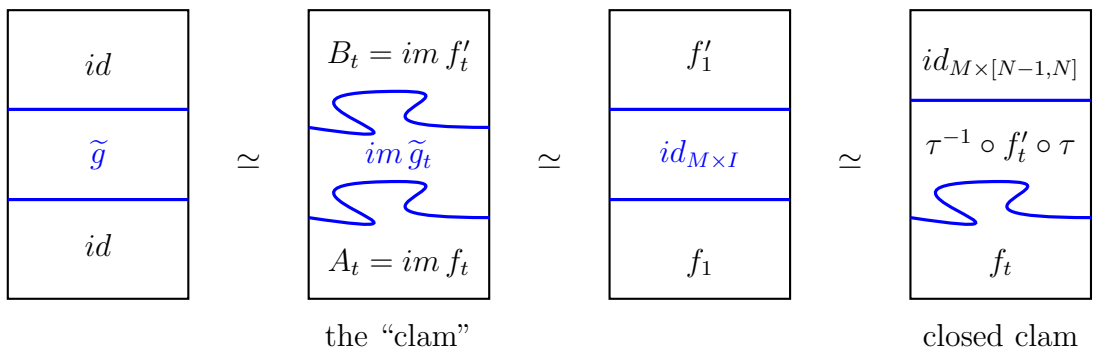

\begin{thm}
The pseudoisotopy $g\cup \varepsilon(g)$ gives a diffeomorphism $\overline g$ of $M\times S^1$ which is pseudoisotopic to the identity but no power of which is isotopic to the identity.
\end{thm}


\end{document}